\documentclass{article}
\usepackage{graphicx} 
\usepackage[utf8]{inputenc}
\usepackage{amsmath, amsfonts, amssymb, amsthm}
\usepackage[english]{babel}
\usepackage[utf8]{inputenc}
\usepackage{ragged2e}
\usepackage{graphicx}
\usepackage{amscd}
\usepackage{txfonts}
\usepackage{latexsym}
\usepackage{fancyvrb}

\usepackage{caption}
\usepackage{subcaption}
\usepackage{makeidx}
\usepackage{authblk}
\usepackage{tcolorbox}
\usepackage{multicol}
\usepackage[titletoc]{appendix}
\usepackage{listings}
\usepackage[labelformat=simple]{subcaption}
\usepackage{xcolor}

\usepackage[color=teal!30!white, bordercolor=teal, textsize=footnotesize,obeyFinal]{todonotes}

\newtheorem{lemma}{Lemma}[section]
\newtheorem{proposition}{Proposition}[section]
\newtheorem{remark}{Remark}[section]
\newtheorem{theorem}{Theorem}[section]

\title{The Hoffman-Singleton manifold}

\author[1]{Daniel Pellicer\thanks{pellicer@matmor.unam.mx}}
\author[1]{Yesenia Villicaña Molina\thanks{yesenia.villicana.molina@gmail.com}}
\affil[1]{Centro de Ciencias Matemáticas, Universidad Nacional Autónoma de México, Morelia, Michoacán,
Mexico}

\date{}

\begin{document}

\maketitle

\abstract{We introduce the Hoffman-Singleton manifold based on some specific subgraph of the Hoffman-Singleton graph. This manifold is motivated in a combinatorial fashion, and it is defined rigorously in geometric terms. We also present a few geometric properties of this manifold.}

\section{Introduction}

Most of our research as mathematicians follows directions that have already been defined, where only one, two or three areas of mathematics are known to be relevant. Many of us find pleasant when we are made aware of unexpected connections between seemingly far away topics of mathematics.

It seems like the most natural way to construct hyperbolic manifolds is as quotients of the hyperbolic space $\mathbb{H}^3$ by a torsion free, Kleinian group. 

In this paper we construct a hyperbolic 3-manifold from a substructure of the famous Hoffman-Singleton graph. This is the only graph with diameter $2$, girth $5$ and valency $7$. Even if in the formal construction of the manifold we pass through a subgroup of a hyperbolic reflection group, the key of the construction is the combinatorial structure of a certain subgraph of the Hoffman-Singleton graph.

The Hoffman-Singleton graph has been object of an extensive study; it is to be expected that some of its interesting combinatorial properties translate into interesting geometric properties of the manifold. For example, many automorphisms of this graph reflect as isometries of the resulting manifold.

The paper is organized as follows. We start by describing the Hoffman-Singleton graph and some of its features in Section \ref{s_graph}. Then, in Section \ref{sec:description} we formally construct a hyperbolic $3$-manifold that was motivated by a substructure of the Hoffman-Singleton graph. In Section \ref{s:properties} we also provide some relevant geometric properties of this manifold.

\section{Construction from the Hoffman-Singleton graph}\label{s_graph}

Before proceeding to the construction of a manifold we recall some basic properties of the Hoffman-Singleton graph (HSG) (see for example \cite{HSG1}, \cite{Godsil}, \cite{HSG2}, \cite{HSG0}).

    \begin{figure}[h]
\begin{center}
    \includegraphics[width=3cm]{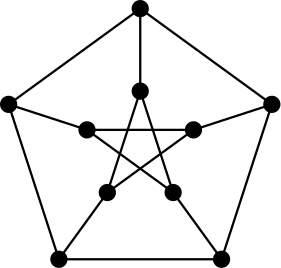}
\end{center}
  \caption{Petersen Graph.}
  \label{fig:Petersen}
\end{figure}

The Petersen graph (see Figure \ref{fig:Petersen}) and the HSG are the only graphs known to exist with diameter $2$, girth $5$ (the {\em girth} is the length of a shortest cycle) and minimum degree at least $3$. In fact, any other such graph must have all its vertices of degree $57$ and it is a long standing problem to determine whether there exist any more.

The HSG is a connected graph with $50$ vertices and $175$ edges. Every vertex has valency $7$ and any pair of vertices that are not adjacent have a neighbor in common. In fact, these properties constitute an alternative equivalent set of properties that completely determine the HSG. It has $252,000$ automorphisms under which all vertices (or edges or arcs) are equivalent.  To highlight how symmetric the HSG is, it suffices to compare with the $10$ vertices, $15$ edges and only $120$ automorphisms of the Petersen graph. 

The properties of the HSG graph imply that it has no triangles and no squares, and that every path with $3$ edges can be completed to a unique pentagon. The following lemma is not so obvious, but it is still true. It can be easily verified either directly (by using the high degree of symmetry of the HSG), or with the help of a computer software.

\begin{lemma}\label{l:PetProp}
Given any pentagon $P$ of the HSG and another edge $e$ incident to a vertex of $P$ then there is a unique Petersen graph contained in the HSG that contains $P$ and $e$.
\end{lemma}

The uniqueness in Lemma \ref{l:PetProp} (in case such a Petersen graph exists) follows easily from the facts that the HSG has diameter $2$ and girth $5$. However, the existence of such a Petersen graph does not seem to be derived directly from these properties.

Given an edge $e=uv$ of the HSG we may label the other neighbors of $u$ and of $v$ by $\{u_1,\dots,u_6\}$ and $\{v_1,\dots,v_6\}$, respectively. If
$$W = V(HSG) \setminus \{u,v,u_i,v_i : i \in \{1,\dots,6\}\}$$
then every $w \in W$ is adjacent exactly to one vertex $u_i$ and to one vertex $v_j$. Conversely, the diameter and girth of the HSG imply that there is exactly one vertex in $W$ that is adjacent to a given vertex $u_i$ and to a given vertex $v_j$. This justifies to label the elements of $W$ by the pairs in $(\mathbb{Z}_6)^2$. The subgraph of the HSG induced by $W$ is elsewhere known as the {\em Sylvester graph} (see for example \cite[Theorem 7.5.3]{Brouwer}).

The edges of the HSG between pairs of vertices in $W$ satisfy the following properties, all of them consequences of the degrees of the vertices together with the diameter and girth of the HSG.
\begin{itemize}
    \item There is no edge between two vertices having the same first entry, or having the same second entry.
    \item Each vertex $(i,j)$ is adjacent to exactly one vertex $(i',j')$ for each $i' \ne i$, and to one vertex $(i'',j'')$ for each $j'' \ne j$.
    \item The subgraph of the HSG induced by $W$ contains no triangles and no squares. 
\end{itemize}

Furthermore, Lemma \ref{l:PetProp} implies the following two results.

\begin{lemma}\label{l:Crosses}
    With the notation above, if there is an edge of the HSG between the vertices $(i,j)$ and $(k,m)$ of $W$ then also $(i,m) (k,j)$ is an edge of the HSG.
\end{lemma}

\begin{proof}
According to Lemma \ref{l:PetProp} there is a unique Petersen graph $H$ containing the pentagon with vertices $u$, $v$, $u_i$, $v_j$ and $(i,j)$ of the HSG as well as the edge $(i,j) (k,m)$. Recall that the Petersen graph has diameter $2$. Then $H$ must contain the unique common neighbors $u_k$ and $v_m$ of $(k,m)$ and $u$, and of $(k,m)$ and $v$, respectively. It also must contain the common neighbor $(k,j)$ of $u_k$ and $v_j$ and the common neighbor $(i,m)$ of $u_i$ and $v_m$. The only missing edge to complete the Petersen graph is $(i,m) (k,j)$ and therefore it must be an edge of the subgraph of HSG.
\end{proof}

The previous lemma can be reinterpreted as follows. Consider the subgraph $G_W$ of the HSG induced by $W$ and order the vertices as a $6 \times 6$ grid. When taking two rows $i$ and $k$ as well as two columns $j$ and $m$, either $G_W$ contains as edges the two diagonals of the rectangle formed by the four intersections of those rows and columns, or $G_W$ contains none of the diagonals as edges (see Figure \ref{fig:crosses}).

    \begin{figure}[h]
\begin{center}
    \includegraphics[width=5cm]{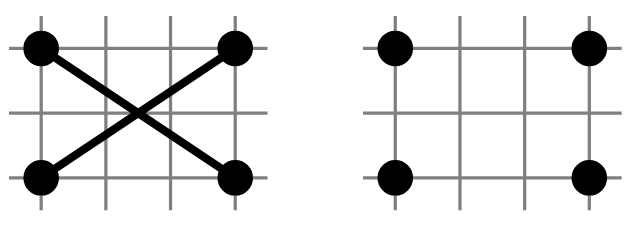}
\end{center}
  \caption{Either the two diagonals of each rectangle or none are part of $G_W$.} 
  \label{fig:crosses}
\end{figure}

\begin{lemma}\label{l:hexagons}
    With the notation above, the subgraph $H_3$ induced by the vertices of three rows of $W$ is the union of three disjoint hexagons. Furthermore, each hexagon has a vertex in each column.
\end{lemma}

\begin{proof}
    When restricting ourselves to three rows of $W$, the properties of the subgraph induced by $W$ forces every vertex to have degree $2$, and therefore $H_3$ is union of cycles.

    Once again we use Lemma \ref{l:PetProp} to find the unique Petersen graph containing the pentagon with vertices $u, u_i, u_j, (i,m_1), (j,m_2)$ and the edge between $u$ and $u_k$ (here $\{i,j,k\}=\{1,2,3\}$ indicate the three rows being considered in $H_3$, while $m_1$ and $m_2$ are some elements of $\{1,\dots,6\}$). It must contain the common neighbors $(k,m_3)$ and $(k,m_4)$ of $u_k$ and $(i,m_1)$ and of $u_k$ and $(j,m_2)$. The two remaining vertices of $H$ must be the neighbor of $(k,m_3)$ in Row $j$ (and so it is also a neighbor of $u_j$), and the neighbor of $(k,m_4)$ in Row $i$ (and so it is also a neighbor of $u_i$). The only way to complete the Petersen graph is if the last two vertices are also adjacent, completing the hexagon. Since we started with any arbitrary edge $(i,m_1)(j,m_2)$, it follows that all connected components of $H_3$ are hexagons. Since $H_3$ has $18$ vertices, all with degree $2$, there must be precisely three connected components.

    The fact that each hexagon has a vertex in each column follows directly from the properties of the subgraph induced by $W$.
\end{proof}

Up to relabeling of the columns, the subgraph $H_3$ in Lemma \ref{l:hexagons} is the one in Figure \ref{fig:hexagons}, where hexagons are distinguished in gray, blue and green. The symmetry properties of the HSG guarantee that any choice of three rows result in the same graph up to relabeling of columns (or equivalently, the neighbors of $v$).

    \begin{figure}[h]
\begin{center}
    \includegraphics[width=8cm]{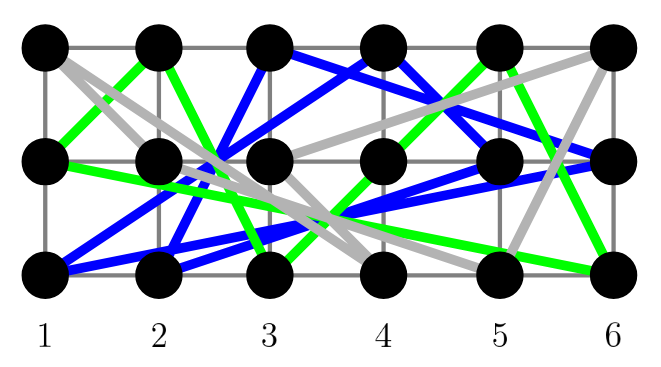}
\end{center}
  \caption{The graph $H_3$ induced by $3$ rows is union of $3$ hexagons.}
  \label{fig:hexagons}
\end{figure}

We now construct a map on a surface in the following way. We take the three hexagons of $H_3$ described in Lemma \ref{l:hexagons}, each with vertex set $\{1,\dots,6\}$ (we only keep the column of each vertex as an element of $W$). We assume them as $2$-dimensional entities with boundary, not just as vertices and edges. Then Lemma \ref{l:Crosses} together with the properties of the subgraph induced by $W$ force every edge to belong to precisely two such hexagons. Gluing the three hexagons along the edges in common we obtain the desired map; the underlying surface is a torus. The map arising from the hexagons in Figure \ref{fig:hexagons} is illustrated in Figure \ref{fig:FirstTorus}; it was denoted by $\{6,3\}_{(1,1)}$ in \cite{Moser}.

    \begin{figure}[h]
\begin{center}
    \includegraphics[height=3.8cm]{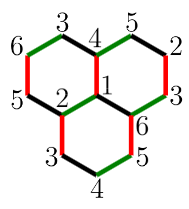}
\end{center}
  \caption{The map on the torus induced by the three hexagons in Figure \ref{fig:hexagons}. Pairs of edges in the boundary with the same vertices are to be identified. Black, red and green edges indicate edges between vertices in the bottom two rows, in the bottom and top row, and in the top two rows, respectively. }
  \label{fig:FirstTorus} 
\end{figure}

When considering the subgraph $H_4$ of $W$ induced by $4$ rows we automatically obtain four isomorphic maps on the torus, all with same vertex set $\{1,\dots,6\}$. Each of these four maps corresponds to a choice of three rows out of the four considered in $H_4$. Every edge $e$ belongs to precisely two maps on the torus, that correspond to the two choices of the third row of that map, besides the two that involve $e$. One of these maps is the one in Figure \ref{fig:FirstTorus}; the other three are illustrated in Figure \ref{fig:ThreeTori} (again, up to relabeling of rows and columns).

    \begin{figure}[h]
\begin{center}
    \includegraphics[height=3.8cm]{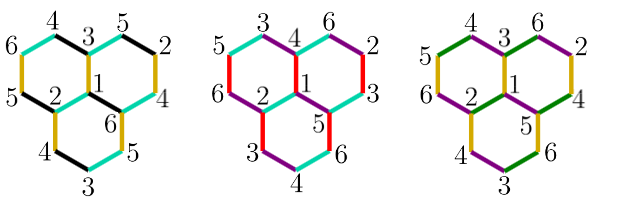}
\end{center}
  \caption{The remaining three maps on the torus induced by four rows of $W$. Blue, yellow and purple edges indicate edges between vertices in the first and fourth rows, in the second and fourth rows, and in the third and fourth rows, respectively.}
  \label{fig:ThreeTori}
\end{figure}

For each triple of rows in $H_4$ we consider the $3$ hexagons described in Lemma \ref{l:hexagons} to obtain a total of $12$ hexagons, each arising from three rows of $H_4$. Each edge of these $12$ hexagons belongs to two hexagons. It follows that again we can build a surface out of these hexagons. This surface turns out to be connected and again to be a torus. We label $X$, $Y$, $Z$ and $W$ the four rows in $H_4$, and label the edges between two rows as well as the hexagons determined by three rows by the juxtaposition of the names of the rows. With this understanding the map with $12$ hexagons is that in Figure \ref{fig:BigHex}, up to relabeling the vertices; it was denoted $\{6,3\}_{(2,2)}$ in \cite{Moser}.

    \begin{figure}[h]
\begin{center}
    \includegraphics[height=5.5cm]{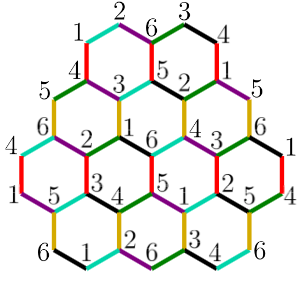}
\end{center}
  \caption{The map on the torus with $12$ hexagons obtained from $H_4$. Colors of the edges are explained in Figures \ref{fig:FirstTorus} and \ref{fig:ThreeTori}.}
  \label{fig:BigHex}
\end{figure}

We have now $5$ maps on the torus with the property that every hexagon belongs to two of them: it belongs to the map induced by the three rows containing its edges, and also to the map with $12$ hexagons. These $5$ maps form a structure of rank $4$ (it has four kinds of elements, namely vertices, edges, hexagons and maps) denoted by `maniplex' in \cite{Maniplexes}. When identifying these $5$ maps along the common hexagons the edges of the map with $12$ hexagons get identified by pairs, and the topological space carrying those $12$ hexagons is no longer a surface. The underlying graph of the entire structure is the complete graph with vertex set $\{1,2,3,4,5,6\}$ with double edges between $1$ and $2$, between $3$ and $4$ and between $5$ and $6$.

Now we have the ingredients to construct a hyperbolic $3$-manifold, which will be called in the following {\it HS-manifold}. Intuitevely, we construct a cusp based on each of these five tori and glue them by pairs along the common hexagons. Formally, we may take the paracompact hyperbolic honeycomb $\{6,3,4\}$. The building blocks of this honeycomb are `tessellations' of horospheres by hexagons (the vertices are in the horosphere but the interiors of the edges are not) where every vertex belongs to percisely three hexagons. The honeycomb is constructed so that every edge belongs to precisely four tessellations of horospheres. The symmetry group of $\{6,3,4\}$ is the Coxeter group $[6,3,4]$ and it acts transitively on the quadruples consisting of incident vertex, edge, hexagon and tessellation. The quotient of a suitable subgroup of $[6,3,4]$ will produce a manifold with $5$ cusps, each over a map on the torus like the ones described above. The details of the honeycomb $\{6,3,4\}$ and of this construction of the HS-manifold will be given in the next section. Alternatively, the HS-manifold can also be understood as the underlying topological space of the corresponding combinatorial map in the sense of \cite[Section 2]{Vince}, or as the order complex of the maniplex consisting of the five tori, in the sense of \cite[Section 2C]{ARP}.


\section{Geometric construction of the HS-manifold}\label{sec:description}

As we mentioned before, in this section we will describe the details of the HS-manifold which was intuitively constructed in the previous section, and we will present some of its geometric properties.

Indeed, let $\mathcal{T}$ be the hyperbolic tetrahedron with an ideal vertex at $\infty$ and its other three vertices at $$v_0=\left( -\frac{\sqrt{2}}{\sqrt{3}}, \, 0 , \, \frac{1}{\sqrt{3}}\right), \hspace{0.4cm} v_1=\left(-\frac{\sqrt{3}}{{2 \sqrt{2}}},\, -\frac{1}{2\sqrt{2}}, \, \frac{1}{\sqrt{2}}\right), \hspace{0.4cm}  \text{and} \hspace{0.4cm}  v_2=(0,0,1).$$  Its dihedral angles are $\angle \mathcal{B}\mathcal{D}=\angle \mathcal{A}\mathcal{B}=\angle \mathcal{A}\mathcal{C}= \pi/2, \,  \angle \mathcal{C}\mathcal{D}=\pi/3, \,\angle \mathcal{A} \mathcal{D}=\pi/4, \,$ and $ \angle \mathcal{B} \mathcal{C}=\pi/6 \, $ (see Figure \ref{fig:tetraedro}), where $\mathcal{A}=\triangle (v_0, v_1, v_2)$, $\mathcal{B}=\triangle (v_1, v_2, \infty)$, $\mathcal{C}=\triangle (v_0, v_2, \infty)$ and $\mathcal{D}=\triangle (v_0, v_1, \infty)$ are the faces of $\mathcal{T}$ (here $\triangle (a,b,c)$ symbolizes the hyperbolic triangle with vertices $a, b$ and $c$). 

\begin{figure}[h]
\begin{center}
    \includegraphics[height=7cm]{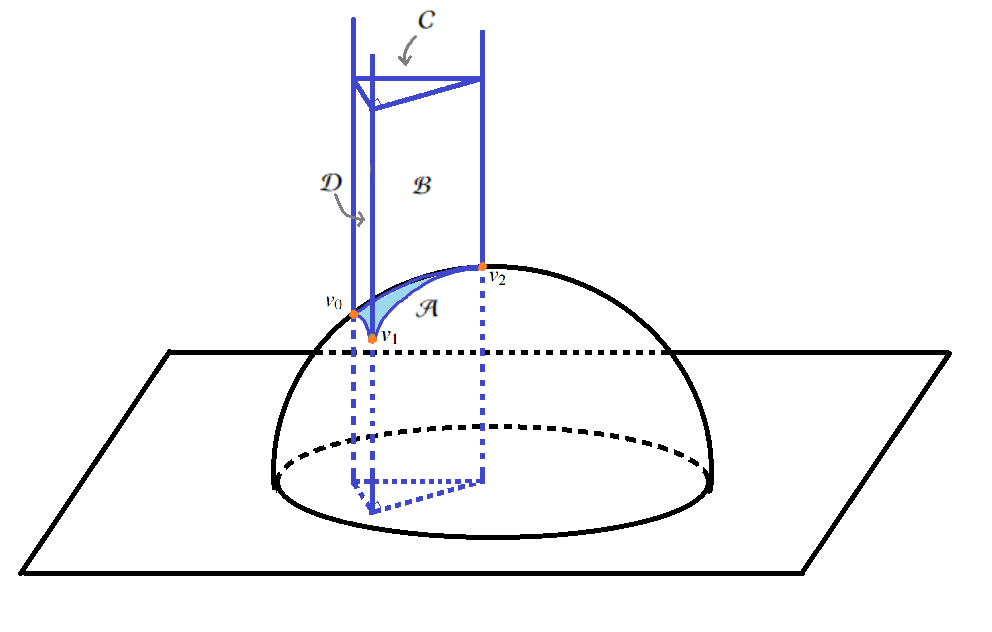}
\end{center}
  \caption{The hyperbolic tetrahedron $\mathcal{T}$.}
  \label{fig:tetraedro}
\end{figure}

We denote by $R_\mathcal{X}$ the reflection with respect to the plane that contains the face $\mathcal{X}$ and we consider the Coxeter group $$[6,3,4]=\langle R_{\mathcal{A}}, R_{\mathcal{B}}, R_{\mathcal{C}}, R_{\mathcal{D}} \rangle.$$ 

 Through $[6,3,4]\mathcal{T},$ we obtain a paracompact tessellation by tetrahedra of the space $\mathbb{H}^3$ (see \cite{Davis}). From this tessellation, we define vertices, edges, hexagons and tessellations of horospheres as follows:
\begin{itemize}
\item {\bf Vertices:} $[6,3,4]v_0$. 
Note that not all the vertices of the tessellation by tetrahedra are vertices using our definition, only those in the orbit of $v_0$.
\vspace{-0.2cm}
\item {\bf Edges:} $[6,3,4]a_0$, where $a_0$ is the geodesic segment between $v_0$ and $R_\mathcal{B}(v_0).$ Note that all edges defined here are the union of two edges of different tetrahedra.
\vspace{-0.2cm}
\item {\bf Hexagons:} $[6,3,4]h_0$, where $h_0$ is the hexagon $\langle R_{\mathcal{B}}, R_{\mathcal{C}} \rangle a_0.$
\end{itemize}

\vspace{-0.4cm}
\begin{itemize}
    \item {\bf Tessellations:} $[6,3,4]t_0$, where
    $t_0=\langle R_{\mathcal{B}}, R_{\mathcal{C}}, R_{\mathcal{D}} \rangle h_0$ is the `tessellation' by hexagons (honeycomb) with vertices in the horizontal horosphere at height 1 (the edges are not in the horosphere, but we shall still abuse notation and call $t_0$ a tessellation of that horosphere).  Note that for each $g \in [6,3,4]$, the element $g(t_0)$ is a tessellation by hexagons of the horosphere that is centered on $g(\infty)$, and that contains all vertices of $g(t_0)$. 
    \vspace{-0.1cm}

\end{itemize}

Given the vertices of a hexagon of $t_0$, we call {\it upward hexagonal cone} the convex hull of the point at infinity and these 6 points, and we call {\it downward hexagonal cone} the convex hull of these 6 points and perpendicular projection in $\mathbb{C}\subset \partial \mathbb{H}^3$ of the center of the hexagon (see Figure \ref{fig:cones}).

\begin{figure}[h!]
  \begin{subfigure}[b]{0.5\textwidth}
    \includegraphics[width=\textwidth, height=\textwidth]{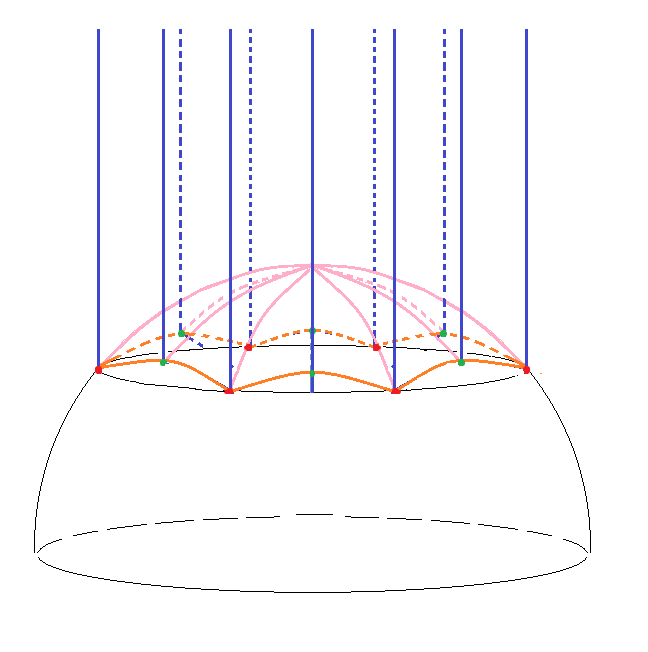}
    \caption{) Upward hexagonal cone.}
    \label{fig:cones1}
  \end{subfigure}
  \hfill
    \begin{subfigure}[b]{0.5\textwidth}
    \includegraphics[width=\textwidth, height=\textwidth]{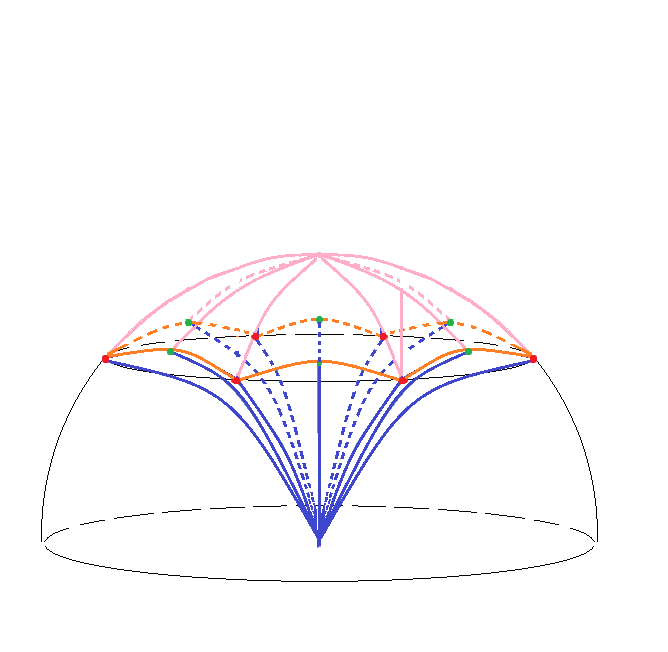}
    \caption{) Downward hexagonal cone.}
    \label{fig:cones2}
  \end{subfigure}
  \caption{}
  \label{fig:cones}
\end{figure}
To construct the HS-manifold, we consider the downward and upward hexagonal cones associated to the 12 hexagons of the map in Figure \ref{fig:BigHex} (see a scale of its geometric realization in Figure \ref{fig:etiquetas}). From now on, we will call $\mathcal{R}$ the union of these cones, and we identify its walls according to the combinatorics described in Section \ref{s_graph}.
\begin{figure}[h!]
\begin{center}
    \includegraphics[width=\textwidth]{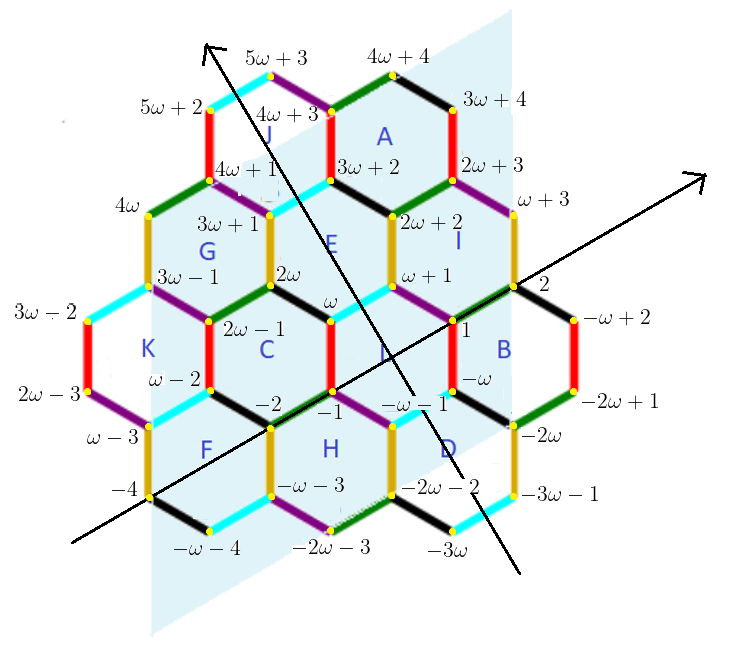}
\end{center}
  \caption{The map of Figure \ref{fig:FirstTorus} with coordinates on a scale of $1:\lambda$ ($=\sqrt{2/3}$), where $\omega=e^{i\pi/3}$ is the first root of unity.} 
  \label{fig:etiquetas}
\end{figure}
Indeed, the identifications of the walls of the upward hexagonal cones are simple, since the walls that are next to each other coincide (and can be understood as identified to each other), and the remaining pairs of boundaries of the upward cones are identified by the three translations $z \mapsto z+6\lambda$, $z\mapsto z+6\lambda\omega$ and $z\mapsto z+6\lambda(1+\omega)$ (one of them is a composition of the other two), where $\lambda= \sqrt{2/3}$ and $\omega=e^{i\pi/3}$ is the first root of unity. 

The identifications in the downward hexagonal cones are a bit more complicated. 
Figure \ref{fig:identif} shows the correspondence between the hexagons with (6)(12)/2=36 identifications of the triangles of their downward cones. Indeed, on the left side four configurations appear and each one has three hexagons. We identify the faces of the three downward hexagonal cones in each configuration as it is indicated by the numbering of the endpoints of each edge. Then, we identify each of these hexagons with its corresponding hexagon in the configuration on the right (see their letter labels). For example, in Figure \ref{fig:pegados} we represent the identifications that come from the configuration of the 3 hexagons labeled with A, B and C as follows: each triangle in Figure \ref{fig:pegados} represents the ideal triangular face of its downward hexagonal cone adjacent to the respective edge, and we identify faces whose respective triangles are the same color (just as their respective edges are identified in the map on the torus in Figure \ref{fig:FirstTorus}).

    \begin{figure}[h!]
\begin{center}
    \includegraphics[width=\textwidth]{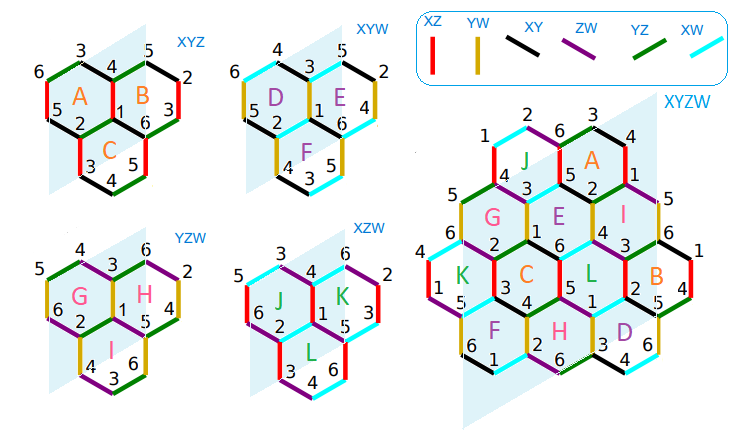}
\end{center}
  \caption{} 
  
  \label{fig:identif}
\end{figure}

These 36 identifications can easily be written as a composition of two inversions on unit spheres, translations by $\lambda$ and $\lambda \omega$ and maybe one rotation by $\pi$ around a vertical geodesic in $\mathbb{H}^3$. For instance, we can identify the ideal triangles associated to the yellow triangles in Figure \ref{fig:pegados} as follows. Specifically, we will send the vertices $(4\lambda (\omega+1), 1/\sqrt{3}), (\lambda (3\omega+4), 1/\sqrt{3})$ and $(3\lambda(\omega+1), 0)$, to the vertices $(\lambda (\omega-2), 1/\sqrt{3}), (-2\lambda , 1/\sqrt{3})$ and $(\lambda(\omega-1),0)$, respectively, through the composition of three functions:

\begin{enumerate}
    \item Translation that sends the hexagon $A$ into the hexagon $F$ $$(z,t) \mapsto (z-3\lambda (\omega+2),t).$$ 
    (It sends the triangle with vertices $(4\lambda(\omega+1), 1/\sqrt{3})$, $(\lambda(3\omega+4),1/\sqrt{3})$ and $(3\lambda(\omega+1), 1)$ to the triangle with vertices $(\lambda(\omega-2),1/\sqrt{3})$, $(-2\lambda,1/\sqrt{3})$ and $(-3\lambda, 1).$)
    
    \item Inversion in the unit sphere  associated to the hexagon $F$. (It sends the triangle with vertices $(\lambda(\omega-2),1/\sqrt{3})$, $(-2\lambda,1/\sqrt{3})$ and $(-3\lambda, 1)$ to the triangle with vertices $(\lambda(\omega-2),1/\sqrt{3})$, $(-2\lambda,1/\sqrt{3})$ and $\infty$.)
    \item Inversion in the unit sphere  associated to the hexagon $C$. (It sends the triangle with vertices $(\lambda(\omega-2),1/\sqrt{3})$, $(-2\lambda,1/\sqrt{3})$ and $\infty$ to the triangle with vertices $(\lambda(\omega-2),1/\sqrt{3})$, $(-2\lambda,1/\sqrt{3})$ and $(\lambda(\omega -1), 0)$.)

\end{enumerate}

Note that this procedure is not sufficient in some cases. However, it is enough to add a rotation by $\pi$ around a vertical geodesic in $\mathbb{H}^3$. For instance in order to identify the brown faces in Figure \ref{fig:pegados}, the translation  $(z, t) \to (z+\lambda(\omega -1), t)$ sends the hexagon $B$ into the hexagon $L$, then the rotation by $\pi$ around a vertical geodesic passing through the center of the hexagon $C$ sends the hexagon $L$  into the hexagon $K$. Finally, to achieve our objective, it is enough to apply the inversion on the sphere associated to $K$ followed by the inversion on the sphere associated to $C$. Any of the 36 identifications mentioned above can be carried out by a procedure analogous to one of the two previous examples. 

Now let's verify that indeed with this construction we obtain a 3-manifold. That is, after these identifications, we must verify that there are balls around the border points of the given fundamental region.
\begin{itemize}
    \item Since the walls of $\mathcal{K}$ are identified two by two and they are contained in hyperbolic planes, we know that such a ball exists around the points inside the walls.
    \item We will study the points inside the edges in two cases: the orange edges and the blue edges in Figure \ref{fig:cones} b).  \begin{itemize}
        \item A fan of 8 tetrahedra intersects at the orange edges: 4 of them correspond to upward hexagonal cones and the other 4 to downward hexagonal cones. (Recall that every edge appears twice in the torus with $12$ hexagons.) Furthermore, the dihedral angles of the tetrahedra that intersect at these edges are all $\pi/4,$ since they are isometric images of the angle $\angle \mathcal{A}\mathcal{D}$ of Figure \ref{fig:tetraedro}. 
        \item At the blue edges a fan of 6 tetrahedra intersects; two for the lower cone of each of the three hexagons in the corresponding configuration with three hexagons. The dihedral angles of the tetrahedra at such edges are $\pi/3,$ since they are isometric images of the angle $\angle \mathcal{C}\mathcal{D}$ of Figure \ref{fig:tetraedro}. Therefore, in both cases, they add up to $2\pi$.
    \end{itemize}
    \item 48 tetrahedra intersect around each vertex and all of them are isometric to $\mathcal{T}$ (see Figure \ref{fig:tetraedro}). The faces of each tetrahedron that coincide at a given vertex are copies of the faces $\mathcal{A}, \, \mathcal{C}, \, $ and $\mathcal{D}$ of Figure \ref{fig:tetraedro}. This is because in the fundamental region $\mathcal{R}$, 4 vertices are identified and in each of the vertices 12 tetrahedra intersect: 4 tetrahedra for each hexagon (two in the upward hexagonal cone and two in the downward hexagonal cone). Finally, it is a known fact that in the tessellation $[6,3,4]$ these same 48 tetrahedra intersect in the same fashion.
    
    \end{itemize}


The group generated by the last 36 isometries and the 3 isometries that we described before to identify all the walls of the upward and downward hexagonal cones, respectively, is the Kleinian group $\mathcal{G}_{HS}$ associated to the HS-Manifold, equivalently, the HS-Manifold is the quotient $\mathbb{H}^3/\mathcal{G}_{HS}$. 

Moreover, it is possible to express 20 of the 36 identifications in the downward hexagonal cones as a composition of the other 16, and one of the three identifications in the upward hexagonal cones as a composition  of the other two. Therefore, $\mathcal{G}_{HS}$ is generated by at most 18 of the 39 isometries discussed above.


  \begin{figure}[h]
\begin{center}
    \includegraphics[width=0.5\textwidth]{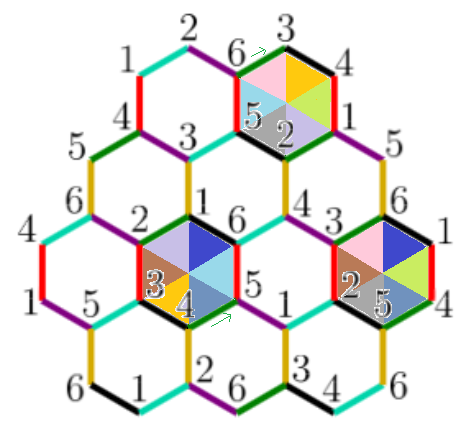}
\end{center}
  \caption{} 
  \label{fig:pegados}
\end{figure}

\begin{figure}[h]
\begin{center}
    \includegraphics[width=\textwidth]{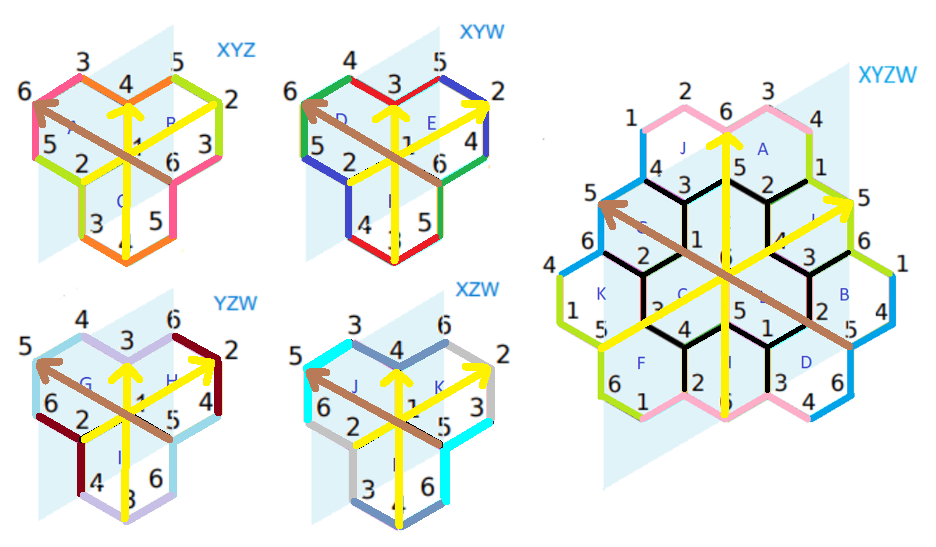}
\end{center}
  \caption{}
  \label{tras}
\end{figure}


\subsection{Some geometric properties of the HS-manifold}\label{s:properties}

{\bf I. Cusps and its volume.} \label{sub_volumen}

Since the HS-manifold is an orientable non-compact hyperbolic 3-manifold of finite volume, it has finitely many disjoint cusps, that is, unbouded ends of finite volume which are diffeomorphic to $T \times (0,\infty)$, where $T$ is a flat torus. The geometry of a cusp is completely determined by the geometry of the flat torus, since the shape of $T \times \{t\}$ does not depend on $t$, that is to say, they are represented by the same element on the moduli space of tori.

In this case, note that in the fundamental region explained in the previous section, there are 5 different orbits in $\partial (\mathbb{H}^3)$, which correspond to the 5 cusps of HS-manifold. Moreover, all of this cusps have the same shape $\langle 1, e^{i \pi/3} \rangle$, although different volume: 4 of them are the {\it small cusps} and they are isometric with each other, and the other one is a {\it big cusp}. In fact the volume of the maximal big cusp is 4 times the volume of any of the maximal small cusps. Indeed, with each configuration of Figure 10 we will construct a cusp, so that the 4 isometric cusps will correspond to the 4 small configurations and the other cusp will correspond to the big configuration.

The large cusp is the quotient of the horoball $\{(z,t)\in \mathbb{H}^3: t>1\}$ by the group of isometries $\langle z \mapsto z+6\lambda, z\mapsto z+6\lambda\omega \rangle$ (the cusp shape is the blue shaded part in the big configuration of Figure \ref{fig:identif}). The above is equivalent to consider the intersection of the horoball $\{(z,t)\in \mathbb{H}^3: t>1\}$ with the fundamental region $\mathcal{R}$ and quotient it with the group $\mathcal{G}_{HS}$. 

We proceed analogously with the small cusps. Each of them corresponds to one of the configurations of three hexagons in Figure 10. For example, we are going to build the cusp associated with the 3 hexagons with labels A, B and C. 
These three hexagons correspond to 3 downward hexagonal cones, whose walls are identified according to the colors of the Figure 11, and therefore, vertical projections  at $ \partial (\mathbb{H}^3) $ of the centers of the three hexagons are identified (by identifying the yellow faces, for example, we identify vertical projections of the centers of the hexagons A and C, and by identifying the pink faces, the centers of the hexagons A and B). In fact, by construction, this cusp is isometric to the intersection of the space $\{(z,t)\in \mathbb{H}^3: t>1 \}$ with the union of the 3 upward hexagonal cones over the 3 hexagons labeled A, B and C, of the small configuration from the upper left corner in Figure \ref{fig:identif}, after identify their walls with the translations $z \mapsto z+3\lambda$ and $z \mapsto z+3\lambda \omega$. 


 \begin{figure}[h]
\begin{center}
    \includegraphics[width=\textwidth]{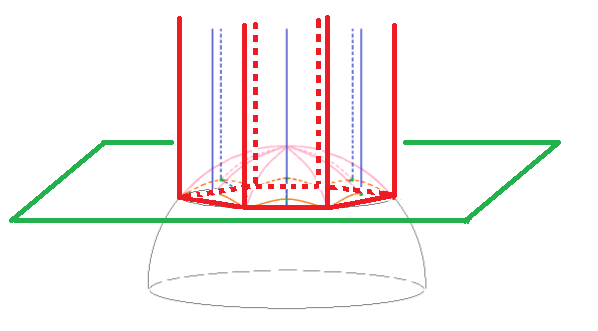}
\end{center}
  \caption{} 
  \label{fig:hexagono_cuspidal}
\end{figure}


The volume of the tetrahedron $\mathcal{T}$ (see Figure \ref{fig:tetraedro}) is well-known (see \cite{RM} or \cite{JMMY}). Since the HS-manifold is the union of 288 tetrahedra, and all of them are isometric with each other, its volume is $$288 \left(-\frac{5}{6} \int_0^{\frac{\pi}{3}} \log |2 \sin t| dt \right) \approx \text{81.1953}. $$


{\bf II. Its full group of isometries.}\label{sub_isometrias} 
Next we determine the full isometry group (isometries that preserve and isometries that reverse orientation). We start by considering the three hyperbolic reflections $R_0$, $R_1$ and $R_2$ by planes perpendicular to the horospheres based at the ideal point of the large cusp. The intersections of $R_0$, $R_1$ and $R_2$ with the configuration of $12$ hexagons are indicated as the green, orange and blue lines in Figure \ref{sim}, respectively.

\begin{figure}[h]
  \begin{subfigure}[b]{0.4\textwidth}
    \includegraphics[width=\textwidth, height=\textwidth]{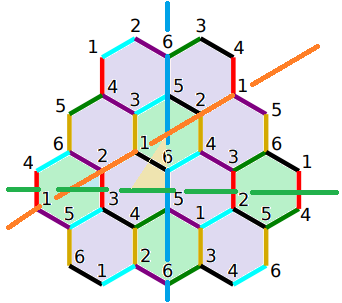}
    \caption{) Three plane reflections in the isometry group.}
    \label{sim1}
  \end{subfigure}
  \hfill
    \begin{subfigure}[b]{0.45\textwidth}
    \includegraphics[width=\textwidth, height=\textwidth]{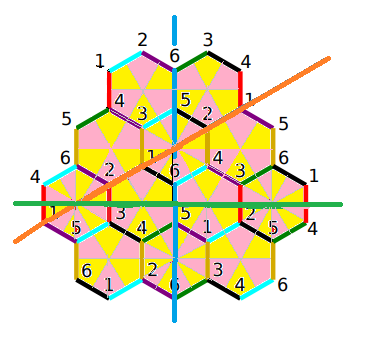}
    \caption{) Fundamental domain of $\langle R_0, R_1, R_2 \rangle$.}
    \label{sim2}
  \end{subfigure}

  \caption{The isometry group of the HS-manifold contains a group isomorphic to $[4,3]$.}
  \label{sim}
\end{figure}

\begin{proposition}\label{prop:subgroup43}
    The group $\langle R_0, R_1, R_2 \rangle$ is a subgroup of the group of isometries of the HS-manifold, and it is isomorphic to the full isometry group of the cube. 
\end{proposition}

\begin{proof}
To verify that each of the reflections is an isometry it suffices to note that it permutes the vertices, the edges, the hexagons and the four small cusps (all of them fix the large cusp). In doing so, it permutes the $288$ ideal tetrahedra from which the maniplex was constructed, since the corners of each tetrahedron are a vertex, a midpoint of edge, a center of hexagon, and the ideal point of a cusp. For instance, the reflection indicated in blue in Figure \ref{sim} induces the permutation $(1,4)(2,3)$ on the vertex set whereas it transposes tori $XYZ$ and $XZW$, and fixes each of tori $XYW$ and $YZW$. The permutations of the edges and of the hexagons can be derived from those of the vertices and tori. In contrast, the hyperbolic reflection whose reflection plane intersects the large configuration in the blue line in Figure \ref{fig:NotRefl} does not induce an isometry of the maniplex, since otherwise vertex 6 should be swapped with vertex 1 and also with vertex 2, which is impossible (consider the images of the pink points in Figure \ref{fig:NotRefl}).

    \begin{figure}[h]
\begin{center}
    \includegraphics[width=6cm]{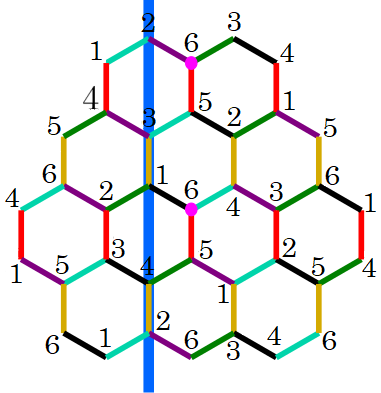}
\end{center}
  \caption{There is no reflection of the configuration along the blue line.}
  \label{fig:NotRefl}
\end{figure}

Next we show that the group $\langle R_0, R_1, R_2 \rangle$ generated by the plane reflections mentioned above is isomorphic to the Coxeter group $[4,3]$ with presentation
$$\langle \rho_0, \rho_1, \rho_0 \, : \, \rho_0^2=\rho_1^2=\rho_2^2=(\rho_0\rho_2)^2=(\rho_0\rho_1)^4=(\rho_1\rho_2)^3 \rangle,$$
which is the full symmetry group of the cube (see \cite{Humphreys}). The angle between the reflection planes $\Pi_1$ and $\Pi_2$ of $R_1$ and $R_2$ is $\pi/3$, and hence $R_1R_2$ has order $3$. Next we consider the halfturn $T=R_0R_2$ about the intersection $\ell$ of the reflection planes of $R_0$ and $R_2$. Clearly, $T$ commutes with $R_2$, and hence the order of $TR_2$ is $2$. Now, the isometry $TR_1$ is a glide reflection about the plane perpendicular to $\Pi_1$ that contains $\ell$ with translation component given by twice the distance from $\Pi_1$ to $\ell$. The order of this glide reflection $TR_1$ is $4$ since the translation by $4$ times the translation component of $TR_1$ is the sum of two of the translations that determine the identifications of the torus. It follows that $\langle R_0, R_1, R_2 \rangle = \langle T, R_1, R_2 \rangle$ is a subgroup of $[4,3]$. On the other hand, the triangle comprised by the reflection lines in Figure \ref{sim} (a) does not tile the region consisting of the $12$ hexagons. A standard procedure shows that the smaller triangles in Figure \ref{sim} (b) are in fact fundamental regions of the group of isometries of the torus generated by those three line reflections. Since there are precisely $48$ small triangles and also $48$ elements in $[4,3]$ we can conclude that $\langle R_0, R_1, R_2 \rangle \cong [4,3]$.
\end{proof}

Now we have a lower bound on the number of isometries of the HS-manifold. Next we give a tool to obtain an upper bound of this number.

\begin{proposition}\label{p:ItFixesStructure}
    Every isometry of the HS-manifold permutes the $288$ ideal tetrahedra described in
Section \ref{sec:description}. Furthermore, every isometry maps a corner of a given type (vertex, midpoint of edge, center of hexagon or ideal point) to a corner of the same type.
\end{proposition}

\begin{proof}
For each cusp $c_i$ ($i \in \{1,2,3,4,5\}$) we may construct a fundamental region $\mathcal{X}_i \subseteq \mathbb{H}^3$ of the Kleinian group $\mathcal{G}_{HS}$ described in Section \ref{sec:description}, where $c_i$ has only one point $P_i$ at infinity. We consider the set of all horospheres centered at $P_i$ whose intersection with $\mathcal{X}_i$ is a topological torus. This set is bounded by a horosphere $\mathcal{H}_i$ whose intersection with $\mathcal{X}_i$ is a quotient of a torus by the identification of only finitely many points. When $c_i$ is the cusp over the large configuration these points are the midpoints of the $18$ edges; each of them appears twice in the configuration (note that these are the only points of the edges in that horosphere). However, if $c_i$ is the cusp over one of the small configurations, say $XYZ$, then the intersection of $\mathcal{X}_i$ with any horosphere containing a vertex or any point of an edge of the configuration is still a topological torus. It only stops being so when the horosphere contains the midpoint of one edge (in fact, of all nine edges) of the large configuration that is not in $XYZ$. From here it follows that
\begin{enumerate}
    \item The cusp over the large configuration must be preserved by every isometry.
    \item The group of isometries must permute the midpoints of the edges when the HS-manifold is constructed as in Section \ref{sec:description}.
\end{enumerate}

Next we take spheres centered at the midpoints of edges, whose radius is half the length of the edge. The $6$ vertices are recovered as the only $6$ points of the manifold that belong to precisely $6$ of those spheres. In fact, each such sphere contains only two of those $6$ points, which are precisely the endpoints of the edge containing the center of the sphere. In this way we recover the vertices and edges. From there the hexagons are recovered easily, since each pair of edges incident at a given vertex with an angle of $\pi/2$ can be extended to a unique planar regular hexagon having angle $\pi/2$ between each pair of consecutive edges.

The previous discussion shows that the isometry group permutes each of the following sets: the set of vertices, the set of midpoints of edges, the set of hexagons (and hence, also the set of their centers), and the set of cusps. This allows us to recover the $288$ ideal tetrahedra in a unique way and the proposition follows.
\end{proof}

We are ready to specify the group of isometries of the HS-manifold.

\begin{theorem}
    The full isometry group of the HS-manifold 
    is isomorphic to the full isometry group of the cube.
\end{theorem}

\begin{proof}
In view of Proposition \ref{prop:subgroup43} we only need to show that the isometry group of the HS-manifold is precisely the group generated by the plane reflections $R_0$, $R_1$ and $R_2$ defined above. Furthermore, we know from Proposition \ref{p:ItFixesStructure} that it must permute the $288$ tetrahedra that were used as building blocks of the HS-manifold.

Given an isometry $R$ and a tetrahedron $\mathcal{T}_0$, the four tetrahedra that share a triangular face with $\mathcal{T}_0$ must be mapped by $R$ to four tetrahedra that share a triangular face with $R\mathcal{T}_0$. The connectivity of the HS-manifold forces the image of each tetrahedron to be uniquely determined by the image of $\mathcal{T}_0$. It follows that the action of the isometry group on the trahedra is free and therefore the number of isometries is a divisor of $288$. Furthermore, every isometry fixes the cusp over the large configuration, and so the $144$ tetrahedra in that cusp must be preserved. This implies that the number of isometries is a divisor of $144$. Proposition \ref{prop:subgroup43} already exhibits $48$ isometries, and from its proof it follows that a tetrahedron cannot be mapped to its image under the reflection indicated with blue in Figure \ref{fig:NotRefl}. Then the number of isometries must be a multiple of $48$ that is a proper divisor of $144$. We can conclude that $\langle R_0, R_1, R_2 \rangle$ is the entire group of isometries.
\end{proof}

\begin{remark}\label{r:Hexagons}
There are 2 orbits of hexagons under the symmetry group. One orbit (green hexagons in Figure \ref{sim1}) contains the $4$ hexagons invariant under $6$ plane reflections, while the other one (purple hexagons in Figure \ref{sim1}) contains the $8$ hexagons invariant only under $3$ plane reflections. The edges of the hexagons in the first orbit constitute one orbit of edges, while the other orbit consists of the edges whose endpoints are vertices of distinct hexagons in the first orbit. The second orbit of edges are precisely the double edges; for example, the two edges between vertices $3$ and $4$ (the black one and the brown one) are not contained in any hexagon in the first orbit.
\end{remark}

\vspace{0.5cm}

{\bf III. Some of its geodesics.} 

Note that the geometry of a plane in $\mathbb{H}^3$ is the same that the geometry of the Poincaré disk. This is easy to see, since there is an isometry between any plane in $\mathbb{H}^3$, and the vertical plane $y=0$ in $\mathbb{H}^3$, which is trivially isometric with $\mathbb{H}^2$. Finally, it is known that $\mathbb{H}^2$ and the Poicaré Hyperbolic Disk $\mathbb{D}^2$ are isometric. Then, each hexagon can be represented in $\mathbb{D}^2$ since each hexagon belongs to a hyperbolic plane in $\mathbb{H}^3$. Furthermore, such a plane is tessellated by congruent hexagons, which is part of the paracompact tessellation that we explained in Section \ref{sec:description}. In that tessellation of the hyperbolic plane there are 4 hexagons around each vertex. Indeed, by reflecting through the planes containing the faces in the orbits of $\mathcal{A}$ and $\mathcal{D}$
that contain the edge $\mathcal{A} \cap \mathcal{D}$, we get copies of the tetrahedron $\mathcal{T}$ that share this edge (see Figure \ref{fig:tetraedro}).
Since $\angle \mathcal{A}\mathcal{D} = \pi/4$, we know that there are 8 copies of $\mathcal{T}$, and hence four copies of $\mathcal{A}$ and four of $\mathcal{D}$; one of the copies of $\mathcal{A}$ is coplanar to the original $\mathcal{A}$. By proceeding similarly with the edge $\mathcal{A} \cap \mathcal{C}$ we obtain 4 copies of $\mathcal{T}$ forcing the only copy of $\mathcal{A}$ to be coplanar to $\mathcal{A}$.
Finally, the spherical angle in $\mathcal{A}$ at the vertex $v_0$ is $\pi/4$ and therefore, after proceeding in an analogous way with the appropriate edges, we get 8 copies of $\mathcal{A}$ all of them coplanar that share the vertex $v_0$. This results in four coplanar hexagons containing $v_0$.

Figure \ref{fig:geodesics} shows the tessellations of a hyperbolic plane centered at a purple hexagon and at a green hexagon of $\{6,3,4\}$ (see Remark \ref{r:Hexagons}). We will call them {\it type I tessellation} and {\it type II tessellation}, respectively. Recall that the graph embedded in the HS-manifold is the complete graph in the vertex set $\{1,\dots,6\}$, with double edges between $1$ and $2$, between $3$ and $4$, and between $5$ and $6$. Purple hexagons alternate between single and double edges, whereas green hexagons only have single edges. Since the hyperbolic line of $\mathbb{H}^3$ containing an edge between preimages under $\mathcal{G}_{HS}$ of vertices $1$ and $2$ (resp. $3$ and $4$ or $5$ and $6$) is tessellated by edges between preimages of vertices $1$ and $2$ (resp. $3$ and $4$ or $5$ and $6$), every hexagon of a tessellation of type I is purple, and every hexagon of a tessellation of type II is green.

\begin{figure}
 \centering
  \subfloat[) Type I]{
   \label{fig_typeI}
    \includegraphics[width=0.5\textwidth]{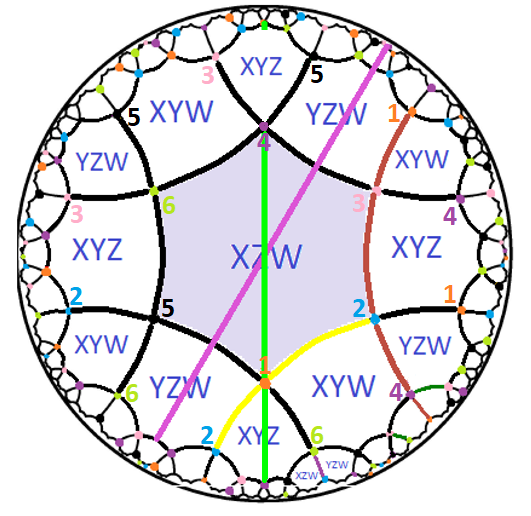}}
  \subfloat[) Type II]{
   \label{fig_typeII}
    \includegraphics[width=0.5\textwidth]{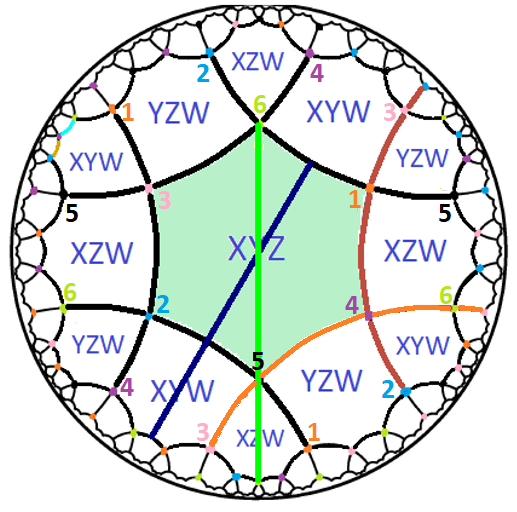}}
\vspace{0.3cm}
 \caption{Some closed simple geodesics of the HS-manifold, viewed in the tessellations of a plane which contains a hexagon.}
 \label{fig:geodesics}
\end{figure}
For the next result, a {\it diagonal} will be a geodesic segment in $\mathbb{D}^2$ with end points at opposite vertices of a hexagon, and a {\it height} a geodesic segment with end points at the midpoints of opposite edges of a hexagon. 
\begin{proposition}\label{p:geodesics}
Some {\bf closed simple} geodesics of the HS-manifold are:
  
  \begin{enumerate}
  
  \item The union of 2 consecutive diagonals which are part of the same geodesic in $\mathbb{D}^2$. The length of such a geodesic is $2 \ln(5+2\sqrt{6})$. 
  \item 
  The union of 2 consecutive heights which are part of the same geodesic in a type II  tessellation of $\, \mathbb{D}^2$. The length of a geodesic of this type is $4 \ln (1+\sqrt{2})$. 

  \item The union of 4 consecutive edges which are part of the same geodesic in a type II tessellation of $\, \mathbb{D}^2$.   
  In the case of type I tessellations, some of them are the union of only 2 consecutive edges (two distinct edges between the same pair of vertices). The lengths of geodesics of this type are $2\ln(2-\sqrt{3})$ or $4\ln(2-\sqrt{3}), $ depending on the type of tessellation.
  \end{enumerate}
  \end{proposition}

\begin{figure}[h]
\begin{center}
    \includegraphics[width=11cm]{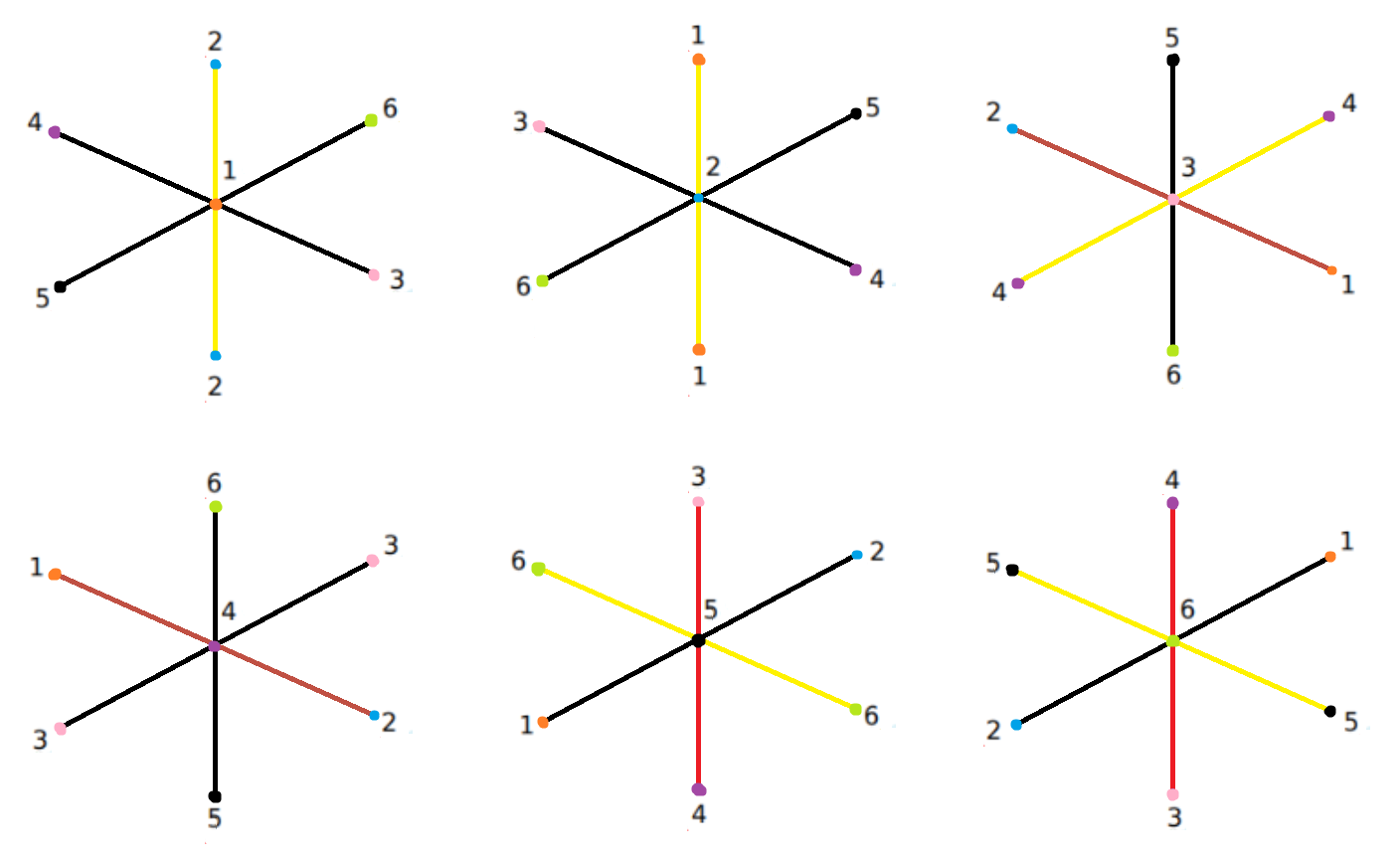}
\end{center}
  \caption{There are 6 vertex orbits, and each of them has the adjacencies shown in this figure. That follows from the combinatorics of Figure \ref{fig:identif}. Furthermore, any three aligned vertices correspond to a geodesic segment in $\mathbb{H}^3$ with vertices labeled with the three numbers appearing on that line, and therefore, any two lines (two pairs of opposite vertices intersecting at another vertex) belong to a hyperbolic plane in $\mathbb{H}^3$. }
  \label{fig:adyacencias} 
\end{figure}

\begin{proof}
As we verified before, there are 6 vertex orbits and the degree of each vertex is 6. Moreover, from Figure \ref{fig:identif} we can determine the labels of all neighboring vertices of each of the vertices (see Figure \ref{fig:adyacencias}), 
and therefore, we can label the vertices of the tessellations in Figure \ref{fig:geodesics}. 

We will focus on the first item for tessellations of type II. Since all diagonals of a hexagon in a tessellation of type II are equivalent under the symmetry group, we only consider one of them. Let $\gamma'$ be the geodesic segment formed by the two green diagonals (see Type II tessellation of Figure \ref{fig:geodesics}), that is the geodesic segment with endpoints $v_{6_1}:=\frac{\sqrt{3}-1}{2}i$ and $v_{6_2}:=\frac{2-6\sqrt{3}}{13}i$ when considered in $\mathbb{D}^2$ (both labeled with the number 6). Let $\gamma$ be the complete geodesic in $\mathbb{D}^2$ that contains $\gamma'$, that is the geodesic with endpoints $i$ and $-i$ in the absolute plane. Let $v_{5_1}:=\frac{1-\sqrt{3}}{2}i$ be the vertex labeled with 5 which is the midpoint between $v_{6_1}$ and $v_{6_2}$.



In the quotient $\mathbb{H}^3/\mathcal{G}_{HS}$ (see Section \ref{sec:description}), we want to prove that the image of $\gamma$ is a closed simple geodesic in the HS-manifold. Indeed 
since $ v_{6_1}$ and $v_{6_2}$ belong to the same orbit, there is a unique isometry $f\in \mathcal{G}_{HS}$ such that $f(v_{6_1})=v_{6_2}$ 
Then, $f(v_{5_1})$ is a vertex at distance $2\ln (1+2/\sqrt{3})$ from $v_{6_2}$ and with label 5. This means that $f(v_{5_1})= v_{5_1}$ or $f(v_{5_1})=v_{5_2}:= \frac{39\sqrt{3}-7}{74}i $ (this is the only neighbor of $v_{6_2}$ with label $5$, other than $v_{5_1}$). Since $f$ cannot have fixed points in $\mathcal{D}^2$ 
we conclude that $f(v_{5_1})=v_{5_2}$.

It follows that $f(\gamma)$ is $\gamma$, and the restriction of $f$ to the plane in consideration must be a loxodromic isometry with fixed points $i$ and $-i$, whose axis of symmetry is $\gamma$.

In the remaining cases we proceed in a completely analogous way. It is only necessary to update the labels of the vertices or midpoints of edges.


Finally, in order to know the lengths of the aforementioned geodesics, it is enough to compute the lengths of a diagonal, a height, and an edge, which are $\,\ln(1+2/\sqrt{3})\,$  $\,2 \ln (1+\sqrt{2}),$ and $\ln(2-\sqrt{3})$ respectively. 
\end{proof}

There are also {\bf non-simple closed} geodesics that are easy to recognize thanks to combinatorics. Examples of these are any union of 4 consecutive heights which are part of the same geodesic in a type I  tessellation of $\, \mathbb{D}^2$. They are closed geodesics with one self-intersection. This is because the ends of any union of two consecutive heights that belong to the same geodesic are in the same equivalence class, and to verify that we need 4 heights to obtain the total length of the geodesic, we justify in a similar way to the procedure followed in the demonstration of the previous proposition. Consider as an example the pink geodesic of the type I tessellation in Figure \ref{fig:geodesics}; its endpoints and its midpoint are in the same equivalence class (all three are midpoints of the same edge between a vertex labeled with 3 and a vertex labeled with 4). The length of a geodesic of this type is $8 \ln (1+\sqrt{2})$.

On the other hand, as we mentioned before, there are 5 different orbits in $\partial(\mathbb{H}^3)$ by the action of the group $\mathcal{G}_{HS}$ which correspond to the 5 cusp of the HS-manifold. We call {\it cusp point} any element in any of these orbits. Then {\bf unicuspid} and {\bf bicuspid} geodesics are easy to define; a bicuspid geodesic is one that is a projection of a geodesic whose ends are cusp points, and a unicuspid geodesic is one that is a projection of a geodesic whose ends are a cuspid point and a non-cusp point.
In our case, the cusp points are the elements in the orbits of the point at infinity or some of the orthogonal projections of the centers of the hexagons (see Figure \ref{fig:etiquetas}). 
Therefore, the vertical geodesics in $\mathbb{H}^3$ that pass through the center of each hexagon are examples of bicuspid geodesics.
 



\section{Conclusions}

Besides its relationship with the HSG, the manifold presented in this work has a rather high group of isometries with respect to its volume. These were two of the original motivations to describe this object.
This first introducton to the HS-manifold describes only few of its properties. Nevertheless, they show that this manifold is interesting from the perspectives of geometry of hyperbolic manifolds, topology of hyperbolic manifolds, group theory related to hyperbolic Coxeter groups, and combinatorics of CW-complexes. We believe that a deeper study will uncover deeper connections among these areas, and perhaps include some others.

\section{Acknowledgments}
This paper was supported by CONACYT ``Fondo Sectorial de Investigaci\'on para la Educaci\'on'' under grant A1-S-10839.


\bibliographystyle{amsplain}
\bibliography{main}

\end{document}